\documentclass[12pt]{amsart}
\usepackage{amssymb, a4wide, mathdots,url,hyperref, shuffle}
\usepackage{bbm}
\usepackage{pb-diagram}
\renewcommand{\subset}{\subseteq}

\newtheorem{theorem}{Theorem}[section]
\newtheorem{lemma}[theorem]{Lemma}
\newtheorem{proposition}[theorem]{Proposition}
\newtheorem{remark}[theorem]{Remark}

\newtheorem{corollary}[theorem]{Corollary}

\title[Identity of parabolic Kazhdan-Lusztig polynomials]{An identity of parabolic Kazhdan-Lusztig polynomials arising from square-irreducible modules}
\author{Maxim Gurevich}
\address{Department of Mathematics, National University of Singapore, 10 Lower Kent Ridge Road, Singapore, 119076}
\email{matmg@nus.edu.sg}

\date{\today}

\newcommand{\gotM}{\mathfrak{m}}
\newcommand{\gotN}{\mathfrak{n}}

\DeclareMathOperator{\irr}{Irr}
\DeclareMathOperator{\seg}{Seg}

\DeclareMathOperator{\multi}{Mult}

\begin{document}

%\renewcommand{\chaptername}{}
%\renewcommand{\thechapter}{}
%\title{Restriction of Arthur-type $GL_n$ representations}
%\title[Quantum invariants for decomposition problems]{Quantum invariants for decomposition problems in type $A$ rings of representations}
%\author{Maxim Gurevich}

\begin{abstract}
We show a precise formula, in the form of a monomial, for certain families of parabolic Kazhdan-Lusztig polynomials of the symmetric group. 

The proof stems from results of Lapid-M\'{i}nguez on irreducibility of products in the Bernstein-Zelevinski ring. By quantizing those results into a statement on quantum groups and their canonical bases, we obtain identities of coefficients of certain transition matrices that relate Kazhdan-Lusztig polynomials to their parabolic analogues.

This affirms some basic cases of conjectures raised recently by Lapid.

\end{abstract}

\maketitle

\section{Introduction}
Given a Coxeter group $W$ and a standard parabolic subgroup $W_J$, the parabolic Kazhdan-Lusztig polynomials $\{P_{\sigma,\omega}(q)\}$ are a collection of integer polynomials attached to each pair of cosets $\sigma,\omega\in W/W_J$. They were initially defined by Deodhar \cite{deod-parab} as a parabolic generalization of ordinary Kazhdan-Lusztig polynomials.

While the ubiquity of Kazhdan-Luzstig polynomials in representation theory is well-established, the parabolic analogues attract a growing interest, fueled in part by discoveries on their geometric nature \cite{kash-parab, anti-sph}. Yet, explicit values or bounds on coefficients of the polynomials are not known in most cases (see, e.g. \cite{brenti-etc} for a discussion of the state of those efforts). %For example, it seems less than a trivial task to determine which polynomials vanish.

In this note, we hope to stress a new approach towards the study of parabolic Kazhdan-Lusztig polynomials, at least for computational purposes. We provide an explicit monomial formula (Theorem \ref{main-thm}) for some polynomials in the case of $W= S_n$ (the symmetric group) and $W_J = S_m \times\cdots \times S_m$, for a divisor $m|n$.

Our approach makes use of the \textit{dual canonical basis} for the quantum group $U_v(\mathfrak{sl}_\infty)^+$. When comparing that basis with another PBW-type basis, it is known by results of Lusztig \cite{lus-canonical} that the values of the transition matrix are given by dimensions of intersection cohomology spaces of certain nilpotent varieties. It is also known that those dimensions can be expressed through coefficients of Kazhdan-Lusztig polynomials (see \cite{Hender} for a survey).

We make an observation (similar to the one made in \cite{burndan-kl}) that the expressions obtained by inverting this matrix (Eq. \eqref{eq-trans-inv}) can be thought of as a \textit{double-coset analogue} of parabolic Kazhdan-Lusztig polynomials. In particular, when the relevant cosets of $W/W_J$ happen to lie in the normalizer of $W_J$, we obtain the parabolic Kazhdan-Lusztig polynomials in the form of transition coefficients (as done in Proposition \ref{coeff-parab}, for our case of interest).

In \cite{LM3}, Lapid-M\'{i}nguez worked on reducibility questions in the so-called Bernstein-Zelevinski ring. They identified conditions under which an irreducible representation $\pi$ of the $p$-adic group $GL_N$ has the property that $\pi\times \pi$ is irreducible as well. Such $\pi$ was called a \textit{$\square$-irreducible} (square-irreducible) representation, a notion intimately related to that of \textit{real} modules, either for quantum affine algebras or KLR algebras (see \cite{kkko0}).

They noted that their irreducibility results can be applied to compute values of certain parabolic Kazhdan-Lusztig polynomials \cite[Corollary 10.9]{LM3} (pertaining to the mentioned groups). Subsequently, Lapid engaged in a computer-assisted exploration which culminated in surprisingly precise conjectures \cite{lapid-conj}, that relate certain parabolic Kazhdan-Lusztig polynomials to corresponding Kazhdan-Luzstig polynomials of the symmetric group.

The above mentioned Theorem \ref{main-thm} proves the basic cases of these conjectures. More precisely, our formula covers a substantial part of those cases in which a suitable ordinary Kazhdan-Lusztig polynomial is trivial. In other words, we show that smoothness properties of Schubert varieties of type $A$ imply simple formulas for certain parabolic Kazhdan-Lusztig polynomials.

We show that the dual canonical basis approach is the missing link which would push the Lapid-M\'{i}nguez irreducibility result into a quantized setting, i.e. extend a statement on the value of the polynomials at $q=1$ into a statement on the full polynomial. The underlying reason for such connection is that the Bernstein-Zelevinski ring can be viewed \cite{LNT},\cite{groj} as a specialization at $q=1$ of the quantum group $U_v(\mathfrak{sl}_\infty)^+$, with the dual canonical basis specializing to the basis of simple modules.

We are unaware at the moment of a geometric explanation of the phenomena indicated in this note. A somewhat different connection between the dual canonical basis and parabolic Kazhdan-Lusztig polynomials was noted in \cite{Frenkel98}. It was also pointed out by Bernard Leclerc that the form of our results bears a curious similarity to those of \cite{leclerc-m}. A direct reason for this remains unclear.

Our hope is that future work with a similar approach can shed more light on the nature of the double-coset sums involved in the transition matrices \eqref{eq-trans-inv}, and the quantum group meaning of the formulas described by Lapid's conjectures \cite{lapid-conj}.

\section{Type $A$ gadgets}

\subsection{Permutations and their associated polynomials}
We will write $S_n$ for the symmetric group of permutations on $\{1,\ldots,n\}$, equipped with the Bruhat order $\leq$. For $x\in S_n$, we let $\ell(x)$ denote its length and $\epsilon(x) = (-1)^{\ell(x)}$ its parity. We denote by $\omega_{0,n}\in S_n$ the longest permutation.

For each pair of permutations $\sigma, \omega\in S_n$ with $\sigma\leq \omega$, we write $P_{\sigma,\omega}\in \mathbb{Z}[q]$ for the corresponding Kazhdan-Lusztig polynomial. It is convenient to write $P_{\sigma,\omega}=0$ as the zero polynomial, when $\sigma\not\leq \omega$.

Suppose that $n=mk$. The group $W_m:= S_m\times\cdots \times S_m$ (product of length $k$) appears as a standard parabolic subgroups of $S_n$. Let $W^m\subset S_n$ be the set of minimal length representatives of the cosets of $S_n/W_m$. 

For a coset $\pi\in S_n/W_m$, we write $\tilde{\pi}\in W^m$ for its representative. We say that $\pi\leq \tau$ for a pair of cosets $\pi,\tau\in S_n/W_m$, when $\tilde{\pi}\leq \tilde{\tau}$.

The quotient group $N_{S_n}(W_m)/W_m$ is naturally identified with $S_k$. We write $r_m:S_k \to N(W_m)/W_m$ for the resulting isomorphism. For a permutation $x\in S_k$, we will also write $t_m(x)=\widetilde{r_m(x)}\in W^m$.

In \cite{deod-parab}, Deodhar defined two polynomials $\hat{P}^q_{\pi,\tau}, \hat{P}^{-1}_{\pi,\tau}\in \mathbb{Z}[q]$, for every pair of cosets $\pi,\tau\in S_n/W_m$ with $\pi\leq \tau$. These are (a particular case of) \textit{parabolic Kazhdan-Lusztig polynomials}.

Although defined in a separate setting, the parabolic Kazhdan-Lusztig polynomials are computationally related to the Kazhdan-Lusztig polynomials, due to the finiteness of $S_n$.

\begin{proposition}\cite{deod-parab}\cite[Proposition 1]{brenti-etc}\label{parab-ident}
 For every $\pi,\tau\in S_n/W_m$ with $\pi\leq \tau$, the identities  
    \[
    \hat{P}^q_{\pi,\tau} = \sum_{x \in W_m} \epsilon(x) P_{\tilde{\pi}x, \tilde{\tau}}\;, \qquad\hat{P}^{-1}_{\pi,\tau} = P_{\tilde{\pi}\omega_m, \tilde{\tau}\omega_m}
     \]
  hold, where $\omega_m$ is the longest element of $W_m$.
\end{proposition}

\subsection{Multisegments and bi-sequences}

Let $\seg$ denote the collection of \textit{segments} of integers, that is, formal pairs $[a,b]$ of integers $a,b$, such that $a\leq b$. 

We say that a pair of segments $[a_1,b_1],[a_2,b_2]\in \seg$ is \textit{in general position}, if $a_1\neq a_2$, $b_1\neq b_2$, $a_1\neq b_2+1$ and $a_2\neq b_1+1$.

We say that a segment $[a_1,b_2]$ \textit{precedes} a segment $[a_2,b_2]$, if $a_1 < a_2$, $b_1<b_2$ and $a_2\leq b_1+1$. In this case we say that the pair $[a_1,b_2],[a_2,b_2]$ is \textit{linked}.

When $\Delta_1= [a_1,b_1]$ precedes $\Delta_2 = [a_2,b_2]$, we naturally write $\Delta_1\cup\Delta_2: = [a_1,b_2]\in \seg$. If $\Delta_1, \Delta_2$ are in addition in general position, we can also write $\Delta_1 \cap \Delta_2: = [a_2,b_1]\in \seg$. 

Let $\multi= \mathbb{Z}_{\geq0}(\seg)$ denote the commutative monoid of \textit{multisegments} of integers, that is, multisets of segments, or more precisely, maps from $\seg$ to $\mathbb{Z}_{\geq0}$ with finite support.

There is a natural embedding of $\seg$ into $\multi$ which sends a segment to its indicator function. Thus, we will sometimes treat segments as elements in $\multi$.

For example, a multisegment $\gotM\in \multi$ can be given as
\[
\gotM = [1,2] + [1,2] + [0,4] + [5,5]\;.
\]

Let us recall the notation which was introduced in \cite{LM3}. 

A pair
\[
\mathcal{A} = \left(\begin{array}{ccc} a_1 & \ldots & a_k \\ b_1 & \ldots & b_k \end{array}\right)
\]
of two sequences of integers satisfying $a_1\leq \ldots \leq a_k$, $b_1\geq  \ldots \geq b_k$ and $a_i \leq b_{k+1-i}+1$, for all $1\leq i\leq k$, will be called a \textit{bi-sequence} of length $k$.

A bi-sequence $\mathcal{A}$ as above defines two standard parabolic subgroups of the group $S_k$, in the following manner. Let $P_1(\mathcal{A})< S_k$ (respectively, $P_2(\mathcal{A})$) be subgroup generated by transpositions $(i\,i+1)$, for which $b_i = b_{i+1}$ (respectively, $a_i = a_{i+1}$) holds.

We write $D(\mathcal{A}) =P_1(\mathcal{A})\setminus S_k / P_2(\mathcal{A})$ for the set of double-cosets related to $\mathcal{A}$. For $\sigma\in D(\mathcal{A})$, we will write $\tilde{\sigma}\in S_k$ for the shortest representative of $\sigma$.

We say that a bi-sequence $\mathcal{A}$ is \textit{regular}, if $P_1(\mathcal{A})$ and $P_2(\mathcal{A})$ are trivial groups. We say that $\mathcal{A}$ as above is \textit{strongly regular}, if it is regular and $\{a_1,\ldots ,a_k\}\cap \{b_1+1, \ldots, b_k+1\} = \emptyset$ holds.

Recall, that the permutation $\sigma_0 = \sigma_0(\mathcal{A})$ is defined for each bi-sequence $\mathcal{A}$ by the following recursion. Given $\sigma_0^{-1}(k), \ldots, \sigma_0^{-1}(i+1)$, we set
\[
\sigma_0^{-1}(i) = \max \{j\not\in \sigma_0^{-1}(\{i+1,\ldots, k\})\;:\; a_j\leq b_i+1\}\;.
\]
By \cite[Section 6.1]{LM3}, for any permutation $\sigma\in S_k$, the inequality $\sigma_0\leq \sigma$ holds, if and only if, $a_i \leq b_{\sigma(i)}+1$, for all $i$.

Given a bi-sequence $\mathcal{A}$ of length $k$ as above and a permutation $\sigma_0(\mathcal{A})\leq \sigma\in S_k$, we can construct a multisegment
\[
\gotM_\sigma(\mathcal{A}) = \sum_{i=1}^k [a_i, b_{\sigma(i)}]\;,
\]
by considering expressions of the form $[b+1,b]$ as empty segments.

Note, that $\gotM_{\sigma'}(\mathcal{A})=\gotM_{\sigma}(\mathcal{A})$, for any $\sigma' \in P_1(\mathcal{A})\sigma P_2(\mathcal{A})$. Thus, we can also write $\gotM_{\overline{\sigma}}(\mathcal{A})$ by specifying a double-coset $\overline{\sigma}\in D(\mathcal{A})$.

In fact, it can be easily seen that every element of $\multi$ can be written in the form $\gotM_\sigma(\mathcal{A})$, for some (non-unique) $\sigma$ and $\mathcal{A}$.

When $\mathcal{A}$ is strongly regular, it is clear that any pair of distinct segments which appears in $\gotM_\sigma(\mathcal{A})$ must be in general position.

Given a bi-sequence of length $k$ and an integer $m\geq1$, we denote by $\mathcal{A}^m$ the bi-sequence of length $n:=mk$ which results in replicating $m$ times each entry of $\mathcal{A}$.

For example,
\[
\mathcal{A} = \left(\begin{array}{ccc} 1 & 2 & 3 \\ 8 & 7 & 6 \end{array}\right),\quad \mathcal{A}^3 = \left(\begin{array}{ccccccccc} 1&1&1 & 2&2&2 & 3&3&3 \\ 8&8&8 & 7&7&7 & 6&6&6 \end{array}\right)\;.
\]

Note, that when $\mathcal{A}$ is a regular bi-sequence, we have $P_1(\mathcal{A}^m)= P_2(\mathcal{A}^m) = W_m$, as subgroups of $S_n$. It is easily seen that in the monoid $\multi$, we have
\begin{equation}\label{double-mult}
m\cdot \gotM_\sigma(\mathcal{A}) = \gotM_{r_m(\sigma)}(\mathcal{A}^m)\;,
\end{equation}
for all $\sigma_0(\mathcal{A})\leq \sigma\in S_k$.

\begin{lemma}\label{lem-213}
Suppose that $\sigma\in S_k$ is a $213$-avoiding permutation, i.e., there are no $1\leq i_1 < i_2 < i_3 \leq k$, for which $\sigma(i_2) < \sigma(i_1) < \sigma(i_3)$ holds.

Then, there exists a strongly regular bi-sequence $\mathcal{A}$, such that $\sigma = \sigma_0(\mathcal{A})$.
\end{lemma}

\begin{proof}
Let us choose a sequence of integers $a_1<< \ldots << a_k << a_{k+1}$ with large enough gaps. Next, we define $b_k,\ldots, b_1$ recursively. Let $b_k$ be an integer satisfying $a_{\sigma^{-1}(k)} < b_k+1  < a_{\sigma^{-1}(k)+1}$.

Supposing $b_k< \ldots < b_{i+1}$ were defined, we set $b_i$ to be an integer, such that $b_{i+1} < b_i$ and $a_j < b_i+1 << a_{j+1}$ are satisfied, for
\[
j = \min \{\sigma^{-1}(i)\leq j'\,:\; b_{i+1}+1 < a_{j'+1}\}.
\]
The pattern avoidance property clearly assures that the algorithm for $\sigma_0(\mathcal{A})$, when $\mathcal{A} = \left(\begin{array}{ccc} a_1 & \ldots & a_k \\ b_1 & \ldots & b_k \end{array}\right)$, produces back the permutation $\sigma$.

\end{proof}

\subsection{Graded nilpotent varieties}

Let us briefly recall a geometric interpretation of the elements of $\multi$. For more details the reader can refer, for example, to \cite[Section 2.3]{LNT}.
 
Consider a (complex) graded vector space $V_{\underline{d}}= \oplus_{k\in \mathbb{Z}}V_k$, with $\dim V_k = d_k$, where $\underline{d} = (d_k)_{k\in \mathbb{Z}_{\geq0}}$ is such that $d_k=0$ except for finitely many values of $k$. Let $E_{V_{\underline{d}}}$ be the space of endomorphisms $x$ of $V_{\underline{d}}$, which satisfy $xV_k \subset V_{k-1}$, for all $k$.

The group $\prod_{k\in \mathbb{Z}}GL(V_k)$ acts on $E_{V_{\underline{d}}}$ by conjugation. The (finite) collection of orbits of that action are in natural correspondence with a certain set of multisegments in $\multi$. Hence, we write each orbit $\mathcal{O}\subset E_{V_{\underline{d}}}$ as $\mathcal{O} = \mathcal{O}_\gotM$, for a unique $\gotM\in \multi$.

Going over all possible spaces $V_{\underline{d}}$ (i.e. varying $\underline{d}$) we obtain a bijection between $\multi$ and the collection of above described orbits in $\{E_{V_{\underline{d}}}\}_{\underline{d}}$. Thus, for each $\gotM\in \multi$, $\mathcal{O}_\gotM$ is a well-defined variety of graded nilpotent operators.

According to \cite{zel-kl}, given a multisegment $\gotM_\sigma(\mathcal{A})\in \multi$, for a bi-sequence $\mathcal{A}$ and (a double-coset of) a permutation $\sigma\in D(\mathcal{A})$, the closure of the corresponding variety $\mathcal{O}_{\gotM_\sigma(\mathcal{A})}$ decomposes as
\[
\overline{\mathcal{O}_{\gotM_\sigma(\mathcal{A})}} = \bigsqcup_{\substack{ \omega\in D(\mathcal{A}), \\ \tilde{\sigma}\leq \tilde{\omega}}} \mathcal{O}_{\gotM_\omega(\mathcal{A})}\;.
\]

Moreover, a Kazhdan-Lusztig polynomial, which is attached to a pair of permutations, is in fact encoded in the singularities of any nilpotent variety that is defined by the pair.

For any $\gotM,\gotN\in \multi$ satisfying the inclusion $\mathcal{O}_\gotN\subset \overline{\mathcal{O}_\gotM}$, it makes sense to denote by $\mathcal{H}^i(\overline{\mathcal{O}_\gotM})_\gotN$ the stalk at a point of $\mathcal{O}_\gotN$ of the $i$-th intersection cohomology sheaf of the variety $\overline{\mathcal{O}_\gotM}$.

\begin{theorem}\cite{zel-2rem}\label{thm-stalk}
For any bi-sequence $\mathcal{A}$ and $\sigma ,\omega \in D(\mathcal{A})$ with $\sigma_0(\mathcal{A})\leq \tilde{\sigma}\leq \tilde{\omega}$, we have
\[
\sum_iq^{i/2}\dim \left(\mathcal{H}^i\left(\overline{\mathcal{O}_{\gotM_\sigma(\mathcal{A})}}\right)_{\gotM_\omega(\mathcal{A})}\right)  = P_{\tilde{\omega}\omega_0, \tilde{\sigma}\omega_0}(q)\;.
\]
\end{theorem}

Given a bi-sequence $\mathcal{A}$ and permutations $\sigma_0(\mathcal{A})\leq \sigma\leq \omega$, let us denote the parameter
\[
c(\sigma,\omega,\mathcal{A}) = \dim \left(\mathcal{O}_{\gotM_{\sigma}(\mathcal{A})}\right) - \dim \left(\mathcal{O}_{\gotM_{\omega}(\mathcal{A})}\right)\;.
\]
\begin{lemma}\label{dim-formula}
For a strongly regular bi-sequence $\mathcal{A}$, permutations $\sigma_0(\mathcal{A})\leq \sigma\leq \omega\in S_k$ and any integer $m\geq 1$, we have
\[
c(r_m(\sigma),r_m(\omega),\mathcal{A}^m) = m^2(\ell(\omega)- \ell(\sigma))\;.
\]
\end{lemma}
\begin{proof}
Let us apply the formula in \cite[Lemma 3.2]{AFK}. It can be used to compute the dimensions of the stabilizers $Stab(x,m)$ of the action (of the same group) on the orbits $\mathcal{O}_{\gotM_{r_m(x)}(\mathcal{A}^m)}$, for $x\in S_k$. 

More precisely, it follows easily from the above mentioned formula and equation \eqref{double-mult} that $\dim Stab(x,m) = m^2(p(x) + \ell(x\omega_0))$, where $p(x)$ here stands for the number of segments in $\gotM_x(\mathcal{A})$. 

Since $\mathcal{A}$ is strongly regular, $p(x)=k$, for all $x\in S_k$. The result then follows by $\ell(\omega)-\ell(\sigma) = \ell(\sigma\omega_{0,k})- \ell(\omega\omega_{0,k})$.

\end{proof}

\section{Quantum groups}
We recall some basic constructions in the theory of quantum groups for the cases that we will require. We will largely follow \cite{LNT}.

Let us fix a number $q = v^{-2}\in \mathbb{C}^\times$ that is not a root of unity.

We set $U_v= U_v(\mathfrak{sl}_\infty)^+$ to be the $\mathbb{Q}(v)$-algbera generated by a sequence of elements $\{E_i\}_{i\in \mathbb{Z}}$, subject to the (quantum Serre) relations
\[
E_iE_j =E_jE_i\quad |i-j|>1\;,
\]
\[
E_i^2E_j - (v+v^{-1})E_iE_jE_i + E_jE_i^2 =0\quad |i-j|=1\;.
\]

\subsection{Bases and their transition matrices}

The algebra $U_v$ comes equipped with the \textit{dual\footnote{It is dual to the \textit{canonical basis} relative to an inner product on the algebra. We will not require these details which can be found in \cite{LNT} and the references therein.} canonical basis} (or, Kashiwara's upper global crystal basis) $\mathcal{B}^\ast= \{G^\ast(\gotM)\}_{\gotM\in \multi}$, which is indexed by the elements of $\multi$.

We can also choose another basis $\mathcal{E}^\ast$ on $U_v$ which serves the role of a dual PBW basis. The elements of $\mathcal{E}^\ast = \{E^\ast(\gotM)\}_{\gotM\in \multi}$ are naturally indexed by multisegments as well as $\mathcal{B}^\ast$, and are constructed through the following induction process (see \cite[Sections 3.2 and 3.5]{LNT}).

We first set $T_\Delta = E^\ast(\Delta) := G^\ast(\Delta)$, for each single segment $\Delta\in \seg$. 

Now, suppose that $\gotM\in \multi$ is given as 
\[
\gotM = m_1\cdot \Delta_1 +\ldots+ m_k\cdot \Delta_k\;,
\]
where $m_1,\ldots,m_k\in \mathbb{Z}_{>0}$ are multiplicities and $\Delta_1 < \ldots < \Delta_k$ are segments ordered according to the relation $<$ defined on $\seg$ as
\[
[a,b] < [c,d] \quad\Leftrightarrow \quad b<d\;\mbox{or}\;\left\{\begin{array}{ll} b=d \\ a< c\end{array}\right.\;.
\]
We then set
\[
E^\ast(\gotM): =  v^{\sum_{i=1}^k{{m_i}\choose{2}}} T^{m_1}_{\Delta_1}\cdot\ldots\cdot T^{m_k}_{\Delta_k}\;.
\]

It was shown in \cite{lus-canonical} (see \cite[Theorem 9]{LNT}) that the transition matrix between the bases $\mathcal{E}^\ast$ and $\mathcal{B}^\ast$ is given as
\begin{equation}\label{eq-trans}
E^\ast(\gotM_\omega(\mathcal{A})) = \sum_{\substack{ \sigma\in D(\mathcal{A}), \\ \sigma_0(\mathcal{A})\leq \tilde{\sigma}\leq \tilde{\omega}}}\left( v^{c(\sigma,\omega, \mathcal{A})} \sum_i v^{-i}\dim \left(\mathcal{H}^i\left(\overline{\mathcal{O}_{\gotM_\sigma(\mathcal{A})}}\right)_{\gotM_\omega(\mathcal{A})}\right)\right) G^\ast(\gotM_\sigma(\mathcal{A}))\;,
\end{equation}
for any bi-sequence $\mathcal{A}$ and $\omega\in D(\mathcal{A})$ with $\sigma_0(\mathcal{A})\leq \tilde{\omega}$.

Using Theorem \ref{thm-stalk}, this is also equivalent to
\begin{equation}\label{eq-trans-kl}
E^\ast(\gotM_\omega(\mathcal{A})) = \sum_{\substack{ \sigma\in D(\mathcal{A}), \\ \sigma_0(\mathcal{A})\leq \tilde{\sigma}\leq \tilde{\omega}}} q^{-\frac{c(\sigma,\omega, \mathcal{A})}2}  P_{\tilde{\omega}\omega_0, \tilde{\sigma}\omega_0}(q)  G^\ast(\gotM_\sigma(\mathcal{A}))\;,
\end{equation}
which is where we see the appearance of Kazhdan-Lusztig polynomials in the quantum group setting.

By using the inversion formulas for Kazhdan-Lusztig polynomials (see e.g. \cite[Eq. (2.11)(2.12)]{Hender}) we can invert the transition matrix to obtain
\begin{equation}\label{eq-trans-inv}
G^\ast(\gotM_\omega(\mathcal{A})) = \sum_{\substack{ \sigma\in D(\mathcal{A}), \\ \sigma_0(\mathcal{A})\leq \tilde{\sigma}\leq \tilde{\omega}}}\left(q^{-\frac{c(\sigma,\omega, \mathcal{A})}2} \epsilon(\tilde{\omega})\sum_{x\in \sigma} \epsilon(x) P_{x, \tilde{\omega}}(q) \right) E^\ast(\gotM_\sigma(\mathcal{A}))\;.
\end{equation}
Note, that the coefficients in \eqref{eq-trans-inv} closely resemble the formulas for the parabolic Kazhdan-Lusztig polynomials $\hat{P}^q$, as in Proposition \ref{parab-ident}. Indeed, let us make that statement precise in our cases of interest. 

Let $\mathcal{A}$ be a regular bi-sequence of length $k$, and $m\geq1$ an integer. Then, elements of $N_{S_{km}}(W_m)/W_m$ can be viewed as \textit{double}-cosets in $D(\mathcal{A}^m)$. Consequently, we deduce the following.

\begin{proposition}\label{coeff-parab}
Given a regular bi-sequence $\mathcal{A}$, permutations $\sigma_0(\mathcal{A})\leq \sigma\leq \omega \in S_k$ and an integer $m\geq1$, 
\begin{enumerate}
  \item The coefficient of the basis element $G^\ast(\gotM_{r_m(\sigma)}(\mathcal{A}^m))$ in the $\mathcal{B}^\ast$-expansion of $E^\ast(\gotM_{r_m(\omega)}(\mathcal{A}^m))$ is given by
\[
q^{-\frac{m^2}2 (\ell(\omega)-\ell(\sigma))}\hat{P}^{-1}_{r_m(\omega\omega_{0,k}),r_m(\sigma\omega_0)}(q)\;.
\]
  \item
  The coefficient of the basis element $E^\ast(\gotM_{r_m(\sigma)}(\mathcal{A}^m))$ in the $\mathcal{E}^\ast$-expansion of $G^\ast(\gotM_{r_m(\omega)}(\mathcal{A}^m))$ is given by
\[
\epsilon(\sigma)^m \epsilon(\omega)^m q^{-\frac{m^2}2 (\ell(\omega)-\ell(\sigma))}\hat{P}^q_{r_m(\sigma),r_m(\omega)}(q)\;.
\]
\end{enumerate}

\end{proposition}

\begin{proof}
The first item follows from \eqref{eq-trans-kl}, Proposition \ref{parab-ident} and  Lemma \ref{dim-formula}, after noting that, for all $\tau\in S_k$, $t_m(\tau \omega_{0,k}) = t_m(\tau)t_m(\omega_{0,k})$ and $t_m(\omega_{0,k})\omega_m= \omega_{0,mk}$.
  
As for the second item, by \eqref{eq-trans-inv}, this coefficient can be expressed as
\[
\epsilon(t_m(\sigma))\epsilon(t_m(\omega)) q^{-\frac{c(\sigma,\omega, \mathcal{A})}2} \sum_{x\in W_m} \epsilon(x)P_{t_m(\sigma)x, t_m(\omega)}(q)\;.
\]
The statement now follows from Proposition \ref{parab-ident}, Lemma \ref{dim-formula} and the fact that $\epsilon(t_m(\tau)) = (-1)^{m^2\ell(\tau)} = \epsilon(\tau)^m$, for all $\tau\in S_k$.

\end{proof}

\begin{remark}
One can give a more general statement from a similar application of \eqref{eq-trans-inv} and an analogue of Proposition \ref{parab-ident}. For any standard parabolic subgroup $P< S_n$, and elements $\sigma,\tau\in N_{S_n}(P)/P$, the corresponding parabolic Kazhdan-Lusztig polynomial $\hat{P}^q_{\sigma,\tau}$ can be read off from a suitable coefficient in the transition matrix from $\mathcal{B}^\ast$ to $\mathcal{E}^\ast$.
\end{remark}

\subsection{Products of basis elements}

While the multiplicative properties of the elements of $\mathcal{B}^\ast$ are not trivial and the subject of ongoing research, multiplication of elements of $\mathcal{E}^\ast$ is easily computable using the commutation relations described in \cite[Proposition 3.11]{bz-strings}.

Let us mention here only the relations for the cases we will require in this note. Namely, suppose that $\Delta_1,\Delta_2\in \seg$ are segments in general position, with $\Delta_1 <\Delta_2$. Then,
\begin{equation}\label{strt-rules}
T_{\Delta_2}T_{\Delta_1} = \left\{\begin{array}{ll} T_{\Delta_1}T_{\Delta_2} & \Delta_1,\Delta_2\mbox{ are not linked,}\\
T_{\Delta_1}T_{\Delta_2} + (v^{-1}-v) T_{\Delta_1 \cap \Delta_2}T_{\Delta_1 \cup \Delta_2} & \Delta_1,\Delta_2\mbox{ are linked.}
 \end{array}\right.
\end{equation}

\begin{proposition}\label{prop1}
Let $\mathcal{A}$ be a strongly regular bi-sequence. Suppose that permutations $\sigma_0(\mathcal{A})\leq \sigma, \omega\in S_k$, a basis element $E\in \mathcal{E}^\ast$ and an integer $m>1$ are given.

Let $c(v)$ be the coefficient of $E^\ast(\gotM_{r_m(\sigma)}(\mathcal{A}^m))$ in the expansion of $E\cdot E^\ast(\gotM_\omega(\mathcal{A}))$ on the basis $\mathcal{E}^\ast$.

Then, $c(v)=0$, unless $E = E^\ast(\gotM_{r_{m-1}(\sigma)}(\mathcal{A}^{m-1}))$ and $\omega=\sigma$.

In the latter case, we have $c(v) = v^{f}$, where $f=f(k,m)$ is a number depending only on $k$ and $m$.

\end{proposition}

\begin{proof}
  Suppose that $E = E^\ast(\gotM)$, for $\gotM = \sum_{i=1}^s \Delta_i\in \multi$, with $\Delta_1\leq \ldots\leq \Delta_s$. Then,
  \begin{equation}\label{eq1}
  EE^\ast(\gotM_\omega(\mathcal{A})) = v^tT_{\Delta_1}\cdots T_{\Delta_s} \cdot T_{[a_{\omega^{-1}(k)},b_k]}\cdots T_{[a_{\omega^{-1}(1)},b_1]}\;,
  \end{equation}
for an integer $t$.

We can use the straightening rules of \eqref{strt-rules} to expand $  EE^\ast(\gotM_\omega(\mathcal{A}))$ on the basis $\mathcal{E}^\ast$.

Suppose that $c(v)\neq0$. Therefore, $E^\ast(\gotM_{r_m(\sigma)}(\mathcal{A}^m))$ is produced through a chain of exchanges from the product \eqref{eq1}, in which a pair $T_{\Delta_j} T_{[a,b_i]}$ with $[a,b_i] < \Delta_j$ is transformed into $T_{\Delta_j \cap [a,b_i]} T_{\Delta_j \cup [a,b_i]}$. In particular, $\Delta_j \cap [a,b_i] = [a',b_i]$, for some $a<a'$.

Ultimately, this chain of exchanges shows that $\gotM_{r_m(\sigma)}(\mathcal{A}^m)$ must contain the segments $[a'_1, b_1],\ldots, [a'_k,b_k]$, for some numbers $a_{\omega^{-1}(i)} \leq a'_i$. Yet, for all $1\leq i\leq k$, the only segments in $\gotM_{r_m(\sigma)}(\mathcal{A}^m)$ whose end-point is $b_i$ are $[a_{\sigma^{-1}(i)},b_i]$. Hence, $a'_i = a_{\sigma^{-1}(i)}$, for all $i$.

In particular, the inequalities $a_{\omega^{-1}(i)}\leq a_{\sigma^{-1}(i)}$ impose the equality $\omega= \sigma$. Now, the equalities $a_{\omega^{-1}(i)} = a'_i$ imply that no transposition was performed. Thus, $\gotM_{r_m(\sigma)}(\mathcal{A}^m) = \gotM + \gotM_\omega(\mathcal{A})$, which clearly gives $\gotM = \gotM_{r_{m-1}(\sigma)}(\mathcal{A}^{m-1})$.

The last statement follows from follows from noting the commutation in \eqref{strt-rules} and the formula in the definition of $E^\ast(\gotM)$. We see that $c(v)$ is given by $v^f$, where $f = k\left( {{{m-1}\choose{2}}} - {{{m}\choose{2}}}\right)$.

\end{proof}

\section{Specialization}
Let $U_{v,\mathbb{Z}}$ be the $\mathbb{Z}[v,v^{-1}]$-algebra spanned by $\mathcal{E}^\ast$ in $U_v$. By considering $\mathbb{C}$ as a $\mathbb{Z}[v,v^{-1}]$-module via $v\mapsto 1$, we are able to construct the \textit{specialization} algebra
\[
A:= \mathbb{C}\otimes_{\mathbb{Z}[v,v^{-1}]}U_{v,\mathbb{Z}}\;.
\]
The set of elements $\mathfrak{B} = \{b^\ast(\gotM):= 1\otimes B^\ast(\gotM)\}_{\gotM\in\multi}$ forms a basis to $A$, which serves as a specialization of the dual canonical basis.

The algebra $A$ together with its basis $\mathfrak{B}$ have a natural categorification in the setting of representations of affine Hecke algebras or $p$-adic groups.

Let us consider the complex affine Hecke algebra $H_{n,q}$ attached the root data of $GL_n$ and the parameter $q$. Let $\hat{R}_n$ be the complexified Grothendieck group of the category of finite-dimensional modules over $H_{n,q}$. Then,
\[
\hat{R}:= \oplus_{n\geq0} \hat{R}_n
\]
has a ring structure, coming from the so-called Bernstein-Zelevinski product. Namely, Given modules $\pi_1,\pi_2$ over $H_{n_1,q},H_{n_2,q}$, respectively, the product $\pi_1\times\pi_2$ is set to be the induction of $\pi_1\otimes \pi_2$ into a module over $H_{n_1+n_2,q}$ using a canonical embedding of algebras $H_{n_1,q}\otimes H_{n_2,q} \subset H_{n_1+n_2,q}$.

In $\hat{R}$, we simply define $[\pi_1]\cdot [\pi_2]:=[\pi_1\times \pi_2]$.

Furthermore, one can work more generally (as Bernstein-Zelevinski originally did \cite{BZ1}) in the setting of smooth complex representations of the groups $GL_n(F)$, where $F$ is a $p$-adic field of residue characteristic $q$ (in case $q$ happens to be an integer prime power). When $R_n$ is set to be the Grothendieck group of this category of representations, we similarly obtain a product structure on
\[
R = \oplus_{n\geq0} R_n
\]
coming from the parabolic induction functor.

It is well known (see discussion in \cite[Section 3.1]{me-restriction}, for example) that there is a natural embedding $\Psi: \hat{R} \hookrightarrow R$ coming from an identification of the Iwahori-invariant block of representations a $p$-adic group with modules over an affine Hecke algebra.

Let $\irr\subset R$ (respectively, $\hat{\irr}\subset \hat{R}$) denote the collection of classes of irreducible objects. Clearly, $\Psi(\hat{\irr})\subset \irr$. 
The collection $\irr$ can be classified in combinatorial terms, through what is known as the Zelevinski classification. In particular, we have an embedding\footnote{Note that most sources would describe $Z$ as a bijection, yet our definition of multisegments in this note was taken to be very restrictive.}
\[
Z: \multi \hookrightarrow \irr\;.
\]
By \cite[Proposition7, Theorem 12]{LNT}, we have a natural algebra embedding
\[
\Phi:A\hookrightarrow \hat{R}\;,
\]
such that $\Phi(\mathfrak{B}) \subset \irr$. 

Moreover, 
\[\Psi\circ \Phi(b^\ast(\gotM)) = Z(\gotM);,
\]
holds\footnote{To be precise, one may need to twist $Z$ by what is known as the Zelevinski involution, depending on chosen conventions. This bears no consequences on our application of results from \cite{LM3}. See further the discussion in \cite[Section 3.2]{me-restriction}.}, for all $\gotM\in\multi$.

\subsection{Square-irreducible modules}

We say that an element $a\in \irr$ is \textit{square-irreducible}, if $a^2\in \irr$. 

By \cite[Corollary 2.7]{LM3},\cite[Corollary 3.4]{kkko0}, when $a\in \irr$ is square-irreducible, $a^m\in \irr$, for all $m\geq1$.

It is well-known (see \cite[Proposition 2.5(5)]{LM2}) that when $Z(\gotM)Z(\gotN)\in \irr$, for $\gotM,\gotN\in\multi$, we must have $Z(\gotM)Z(\gotN) = Z(\gotM+\gotN)$. Hence, by \eqref{double-mult}, for a bi-sequence $\mathcal{A}$ and a permutation $\sigma_0(\mathcal{A})\leq \sigma$, such that $Z(\gotM_\sigma(\mathcal{A}))$ is square-irreducible, we have
\[
Z(\gotM_\sigma(\mathcal{A}))^m = Z(\gotM_{r_m(\sigma)}(\mathcal{A}^m))\;,
\]
for all $m\geq1$.

The work of Lapid-M\'{i}nguez in \cite{LM3} established a characterization of the square-irreducible representations in a certain subset of $\irr$ in terms of bi-sequences and permutations.

\begin{theorem}\cite[Theorem 1.2]{LM3}\label{thm-lm}
Let $\mathcal{A}$ be a regular bi-sequence, and $\sigma_0(\mathcal{A})\leq \sigma$ a permutation.
  
Then, $Z(\gotM_\sigma(\mathcal{A}))$ is square-irreducible, if and only if, the polynomial $P_{\sigma_0(\mathcal{A}), \sigma}$ is trivial, i.e. equals to the constant $1$.  
\end{theorem}

\subsection{Quantum group consequences}

\begin{proposition}\label{prop-mtimes}
  For a multisegment $\gotM\in \multi$, such that $Z(\gotM)$ is square-irreducible, we have
  \[
  G^\ast(\gotM)^m = v^{e(\gotM,m)} G^\ast( m\cdot \gotM)\;,
  \]
for all $m\geq 1$, and a integer $e(\gotM,m)$.

If, in addition, $\gotM = \gotM_\sigma(\mathcal{A})$, for a strongly regular bi-sequence $\mathcal{A}$, then $e(\gotM,m) = e(k,m)$ depends only on $m$ and the length $k$ of $\mathcal{A}$.
\end{proposition}

\begin{proof}
Using the embedding $\Psi\circ\Phi$, we see that $b^\ast(\gotM)^m = b^\ast(m\cdot \gotM)$. The first statement follows from \cite[Proposition 15]{LNT}.

For the second statement, note that $e(\gotM,m)$ can be computed explicitly using \cite[4.2(9)]{LNT}. In particular, we can see from the mentioned formula that when the segments in $\gotM$ all appear with multiplicity $1$ and in general position, $e(\gotM,m)$ depends only on $m$ and the number of segments in $\gotM$. These conditions are satisfied in $\gotM_\sigma(\mathcal{A})$, when $\mathcal{A}$ is strongly regular. Moreover, the number of segments in such $\gotM_\sigma(\mathcal{A})$ always equals the length of $\mathcal{A}$.

\end{proof}

%We adopt the notations \cite{LM3} for multisegments, bi-sequences and their attached permutations, and the notations of \cite{LNT} for quantum groups and their bases.

%Recall that the $\mathbb{C}(q^{1/2})$-algebra $U_v(\mathfrak{sl}_\infty)^+$ ($v=q^{-1/2}$) is equipped with a dual canonical basis $\mathcal{G} = \{G^\ast(\gotM)\}_{\gotM\in \multi}$ and (a choice of) dual-PBW basis $\mathcal{E} = \{E^\ast(\gotM)\}_{\gotM\in \multi}$, both of which parameterized by multisegments of integer intervals.

%By \cite[3.4]{LNT}, the transition coefficients are given as
%\[
%E^\ast(\gotM_\sigma(\mathcal{A})) = \sum_{\sigma'} q^{-\frac{c(\sigma,\sigma',\mathcal{A})}2} P_{\sigma\omega_0,\sigma'\omega_0}(q) G^\ast(\gotM_{\sigma'}(\mathcal{A}))\;,
%\]
%where $\sigma'$ runs over all maximal length representatives of the double cosets defined by the parabolic subgroups defined by $\mathcal{A}$, 

%Inverting, we see that
%\begin{equation}\label{eq3}
%G^\ast(\gotM_\sigma(\mathcal{A})) = \sum_{\sigma_0(\mathcal{A})\leq \sigma'\leq \sigma} q^{-\frac{c(\sigma,\sigma',\mathcal{A})}2} \epsilon(\sigma\sigma')P_{\sigma',\sigma}(q) E^\ast(\gotM_{\sigma'}(\mathcal{A}))\;.
%\end{equation}

%Given a bi-sequence $\mathcal{A}$ and $m\geq1$, we write $\mathcal{A}^m$ for the bi-sequence resulting in adjoining $\mathcal{A}$ with itself $m$ times. For a permutation $\sigma$, we let $\tilde{\sigma}^m$ to be the operation described in \cite{lapid-conj}.

\begin{theorem}\label{main-thm}
  Suppose that $P_{\sigma_0,\omega}(q)\equiv 1$, for some permutations $\sigma_0\leq \omega\in S_k$, such that $\sigma_0$ is $213$-avoiding (in the sense of Lemma \ref{lem-213}). 
  
  Then, for every $\sigma_0\leq  \sigma\leq \omega$ and every integer $m>1$,
  \[
  \hat{P}^q_{r_m(\sigma),r_m(\omega)} (q)  = q^{{{m}\choose{2}} \left(\ell(\omega)-\ell(\sigma)\right)}\;
  \]
holds.

\end{theorem}

\begin{proof}
  
  Let $\mathcal{A}$ be a strongly regular bi-sequence, such that $\sigma_0 = \sigma_0(\mathcal{A})$, which exists by Lemma \ref{lem-213}.

By Theorem \ref{thm-lm} and Proposition \ref{prop-mtimes}, we have
\begin{equation}\label{eq2}
q^{e} G^\ast(\gotM_{r_m(\omega)}(\mathcal{A}^m) ) = G^\ast( \gotM_\omega(\mathcal{A}))^m \;,
\end{equation}
for a number $e=e(k,m)$.

%Recall that $E^\ast(\gotM)$ appears with coefficient $1$ in $G^\ast(\gotM)$ and vice versa, for all $\gotM$. Hence, $s(q)$ must equal to the coefficient of $E^\ast(   \gotM_{\tilde{\omega}^m}(\mathcal{A}^m))$ in the $\mathcal{E}$-expansion of $E^\ast( \gotM_\omega(\mathcal{A}))^m$. By inductively applying Proposition \ref{prop1}, we see that $s(q) = q^e$, where $e= e(n,m)$ depends solely on $n$ and $m$.

Let us compute the coefficient of $E^\ast( \gotM_{r_m(\sigma)}(\mathcal{A}^m))$ in the $\mathcal{E}^\ast$-expansion of both sides of \eqref{eq2}. 

On the left hand side, Proposition \ref{coeff-parab}(2) gives the value
\begin{equation}\label{r-side}
  \epsilon(\sigma)^m \epsilon(\omega)^m q^{e-\frac{m^2}2 (\ell(\omega)-\ell(\sigma))}\hat{P}^q_{r_m(\sigma),r_m(\omega)}(q)\;
\end{equation}
for that coefficient.

%q^{e-\frac{c(\tilde{\omega}^m,\tilde{x}^m,\mathcal{A}^m)}2}\epsilon(\tilde{\omega}^m\tilde{x}^m) \sum_{u\in H}\epsilon(u) P_{\tilde{x}^m u ,\tilde{\omega}^m}(q)\;.
In order to compute the expansion of the right hand side, we note that by \eqref{eq-trans-inv},
\[
G^\ast( \gotM_\omega(\mathcal{A}))^m = 
\]
\[
=\epsilon(\omega)^m\sum_{\sigma_0(\mathcal{A})\leq \sigma_1,\ldots,\sigma_m\leq \omega} q^{-\frac{1}2\sum_{i=1}^m c(\sigma_i,\omega,\mathcal{A})} \epsilon(\sigma_1\cdots \sigma_m) E^\ast(\gotM_{\sigma_1}(\mathcal{A}))\cdots E^\ast(\gotM_{\sigma_m}(\mathcal{A}))\;.
\]
An inductive application of Proposition \ref{prop1} shows that the only summand in the sum above which contributes to the $E^\ast( \gotM_{r_m(\sigma)}(\mathcal{A}^m))$-coefficient of the $\mathcal{E}^\ast$-expansion of $G^\ast( \gotM_\omega(\mathcal{A}))^m$ is
\[
q^{-\frac{m}2 c(\sigma, \omega,\mathcal{A})}\epsilon(\omega)^m\epsilon(\sigma)^m E^\ast( \gotM_\sigma(\mathcal{A}))^m\;.
\]

Recall that $E^\ast(\gotM)$ appears with coefficient $1$ in $G^\ast(\gotM)$ and vice versa, for all $\gotM\in\multi$. Thus, \eqref{eq2} also implies that $E^\ast(\gotM_{r_m(\omega)}(\mathcal{A}^m) )$ appears with coefficient $q^e$ in the $\mathcal{E}^\ast$-expansion of $ E^\ast( \gotM_\omega(\mathcal{A}))^m$. Yet, by Proposition \ref{prop1}, the value of the last mentioned coefficient is independent of $\omega$. Hence, we see that $q^e$ is also the $E^\ast(\gotM_{r_m(\sigma)}(\mathcal{A}^m) )$-coefficient in the $\mathcal{E}^\ast$-expansion of $ E^\ast( \gotM_\sigma(\mathcal{A}))^m$.

Thus, the coefficient is in the right hand side of \eqref{eq2} is given as
\begin{equation}\label{l-side}
q^{e-\frac{m}2 c(\sigma, \omega,\mathcal{A})}\epsilon(\omega)^m\epsilon(\sigma)^m\;.
\end{equation}

Finally, since $c(\sigma,\omega,\mathcal{A}) = \ell(\omega) -\ell(\sigma)$ (Lemma \ref{dim-formula}), a comparison of \eqref{l-side} with \eqref{r-side} gives the desired result.

\end{proof}

\begin{corollary}
Let $\omega\in S_k$ be a permutation. Let $X_\omega$ be the Schubert variety attached to $\omega$, when $S_k$ is viewed as the Weyl group of $SL_k$.

If $X_\omega$ is a smooth variety, then
  \[
  \hat{P}^q_{r_m(\sigma),r_m(\omega)} (q)  = q^{{{m}\choose{2}} \left(\ell(\omega)-\ell(\sigma)\right)}\;
  \]
holds, for all $\sigma\leq \omega$ and integer $m>1$.

\end{corollary}
\begin{proof}
Smoothness of $X_\omega$ implies that $\sigma_0$ can be taken as the identity permutation in Theorem \ref{main-thm}.
\end{proof}

\bibliographystyle{alpha}
\bibliography{propo2}{}

\end{document}